\documentclass[12pt, reqno]{amsart}
\usepackage{amsmath, amsthm, amscd, amsfonts, amssymb, graphicx, color}
\usepackage[bookmarksnumbered, colorlinks, plainpages]{hyperref}
\hypersetup{colorlinks=true,linkcolor=red, anchorcolor=green, citecolor=cyan, urlcolor=red, filecolor=magenta, pdftoolbar=true}

\newtheorem{theorem}{Theorem}[section]
\newtheorem{lemma}[theorem]{Lemma}

\newtheorem{corollary}[theorem]{Corollary}
\theoremstyle{definition}
\newtheorem{definition}[theorem]{Definition}

\theoremstyle{remark}

\numberwithin{equation}{section}

\begin{document}
\setcounter{page}{1}

\title[Construction of continuous K-g-Frames]{Construction of continuous K-g-Frames in Hilbert $C^{\ast}$-Modules}
	
\author[A. Karara, M. Rossafi,  M. Klilou, S. Kabbaj]{Abdelilah Karara$^{1}$, Mohamed Rossafi$^{*2}$,  Mohammed Klilou$^{2}$ and Samir Kabbaj$^{1}$}
	
\address{$^{1}$Department of Mathematics, Faculty of Sciences, University of Ibn Tofail, Kenitra, Morocco}
\email{\textcolor[rgb]{0.00,0.00,0.84}{abdelilah.karara@uit.ac.ma; samkabbaj@yahoo.fr}}
	
\address{$^{2}$Department of Mathematics Faculty of Sciences Dhar El Mahraz, University Sidi Mohamed Ben Abdellah, Fez, Morocco}
\email{\textcolor[rgb]{0.00,0.00,0.84}{rossafimohamed@gmail.com; mohammed.klilou@usmba.ac.ma}}



\date{
	\newline \indent $^{*}$Corresponding author}

\begin{abstract}
In this work, we provide some constructions and the sum of new continuous K-g-frames in
Hilbert$C^{\ast}$-Modules. We provide certain necessary and sufficient conditions for some adjointable
operators on $\mathcal{H}$, under which new continuous K-g-frames can be retrieved from those that already exist. Additionally, we discuss the sum of continuous K-g-frames, discover some of their
characterizations, and offer some adjointable operators to construct new continuous K-g-frames
from the previous ones.

\end{abstract} \maketitle

\section{Introduction and preliminaries}
Frame theory is an active topic of mathematical research in fields such as signal processing,
computer science, and more. Frames for Hilbert spaces were first introduced in \textbf{1952} by Duffin
and Schaefer \cite{Duf} for the study of nonharmonic Fourier series. Daubechies, Grossmann, and
Meyer \cite{Dau} later revised and developed them in \textbf{1986}, and they have been popularized since then.

Recently, many mathematicians have generalized frame theory from Hilbert spaces to Hilbert $C^*$-
modules. For detailed information on frames in Hilbert $C^*$-modules, we refer to \cite{r4, Kho2, r1,Jin}.
Currently, the study of continuous K-g-frames has yielded many results, which were introduced by
Alizadeh, Rahimi, Osgooei, and Rahmani \cite{Aliz}. The study of some of their properties has been
further explored in \cite{Ches}.

Throughout this paper $(\Omega,\nu)$ be a measure space, let $\mathcal{H}$ and $\mathcal{K}$ be two Hilbert $C^{\ast}$-modules, $\{\mathcal{K}_{\xi}: \xi\in\Omega\}$  is a sequence of subspaces of $\mathcal{K}$, we also reserve the notation $\operatorname{End}_\mathcal{A}^*\left(\mathcal{H}, \mathcal{K}_\xi\right)$ for the collection of all adjointable $\mathcal{A}$-linear maps from $\mathcal{H}$ to $\mathcal{K}_{\xi}$ and $\operatorname{End}_\mathcal{A}^*(\mathcal{H}, \mathcal{H})$ is denoted by $\operatorname{End}_\mathcal{A}^*(\mathcal{H})$. We will use $\mathcal{N}(\Theta)$ and $\mathcal{R}(\Theta)$     for the null  and range space of an operator
  $\Theta\in \operatorname{End}_\mathcal{A}^*\left(\mathcal{H}, \mathcal{K}\right)$, respectively.
We also denote 
$$\bigoplus_{\xi \in \Omega} \mathcal{K}_{\xi}=\left\{\alpha=\left\{\alpha_\xi\right\}: \alpha_\xi \in \mathcal{K}_{\xi}\right. \operatorname{and}\; \left\|\int_{\Omega}|\alpha_{\xi}|^{2}\mathrm{d} \nu(\xi)\right\|<\infty\left.\right\}.$$
Let $f=\left\{f_\xi\right\}_{\xi \in \Omega}$ and $g=\left\{g_\xi\right\}_{\xi \in \Omega}$,  the inner product is defined by $\langle f, g\rangle=\int_{\Omega}\left\langle f_\xi, g_\xi\right\rangle \mathrm{d} \nu(\xi)$, we have $\bigoplus_{\xi \in \Omega} \mathcal{K}_{\xi}$ is a Hilbert $\mathcal{A}$-module.
\begin{definition}\cite{Con}.
	Let $\mathcal{A}$ be a Banach algebra, an involution is a map $ x\rightarrow x^{\ast} $ of $\mathcal{A}$ into itself such that for all $x$ and $y$ in $\mathcal{A}$ and all scalars $\alpha$ the following conditions hold:
	\begin{enumerate}
		\item  $(x^{\ast})^{\ast}=x$.
		\item  $(xy)^{\ast}=y^{\ast}x^{\ast}$.
		\item  $(\alpha x+y)^{\ast}=\bar{\alpha}x^{\ast}+y^{\ast}$.
	\end{enumerate}
\end{definition}
\begin{definition}\cite{Con}.
	A $C^{\ast}$-algebra $\mathcal{A}$ is a Banach algebra with involution such that :$$\|x^{\ast}x\|=\|x\|^{2}$$ for every $x$ in $\mathcal{A}$.
\end{definition}
\begin{definition}\cite{Kap}.
	Let $ \mathcal{A} $ be a unital $C^{\ast}$-algebra and $\mathcal{H}$ be a left $ \mathcal{A} $-module, such that the linear structures of $\mathcal{A}$ and $\mathcal{H}$ are compatible. $\mathcal{H}$ is a pre-Hilbert $\mathcal{A}$-module if $\mathcal{H}$ is equipped with an $\mathcal{A}$-valued inner product $\langle.,.\rangle :\mathcal{H}\times \mathcal{H}\rightarrow\mathcal{A}$, such that is sesquilinear, positive definite and respects the module action. In the other words,  
	\begin{itemize}
		\item [(i)] $ \langle x,x\rangle\geq0 $ for all $ x\in \mathcal{H}$ and $ \langle x,x\rangle=0$ if and only if $x=0$.
		\item [(ii)] $\langle ax+y,z\rangle=a\langle x,z\rangle+\langle y,z\rangle$ for all $a\in\mathcal{A}$ and $x,y,z\in \mathcal{H}$.
		\item[(iii)] $ \langle x,y\rangle=\langle y,x\rangle^{\ast} $ for all $x,y\in \mathcal{H}$.
	\end{itemize}	 
For $x\in \mathcal{H}$, we define $||x||=||\langle x,x\rangle||^{\frac{1}{2}}$. where $||.||$  is a norm on $\mathcal{H}$ and if $\mathcal{H}$ is complete this norm we will call it a Hilbert $\mathcal{A}$-module or a Hilbert $C^{\ast}$-module over $\mathcal{A}$. For every $a$ in $C^{\ast}$-algebra $\mathcal{A}$, we have $|a|=(a^{\ast}a)^{\frac{1}{2}}$ 
\end{definition}
\begin{lemma} \cite{Pas}.
	Let $\mathcal{H}$ be Hilbert $\mathcal{A}$-module. If $\mathcal{T}\in End_{\mathcal{A}}^{\ast}(\mathcal{H})$, then $$\langle \mathcal{T}x,\mathcal{T}x\rangle\leq\|\mathcal{T}\|^{2}\langle x,x\rangle, \forall x\in\mathcal{H}.$$
\end{lemma}
\begin{lemma} \cite{Ara}.
Let $\mathcal{H}$ and $\mathcal{K}$ two Hilbert $\mathcal{A}$-modules and $\mathcal{T}\in End^{\ast}_{\mathcal{A}}(\mathcal{H},\mathcal{K})$. Then the following statements are equivalent:
\begin{itemize}
	\item [(i)] $\mathcal{T}$ is surjective.
	\item [(ii)] $\mathcal{T}^{\ast}$ is bounded below with respect to norm, i.e., there is $m>0$ such that $\|\mathcal{T}^{\ast}x\|\geq m\|x\|$ for all $x\in\mathcal{K}$.
	\item [(iii)] $\mathcal{T}^{\ast}$ is bounded below with respect to the inner product, i.e., there is $m'>0$ such that $\langle \mathcal{T}^{\ast}x,\mathcal{T}^{\ast}x\rangle\geq m'\langle x,x\rangle$ for all $x\in\mathcal{K}$.
\end{itemize}
\end{lemma}
\begin{definition} \cite{Kar}
Let $\mathcal{T} \in \operatorname{End}_\mathcal{A}^*(\mathcal{H},\mathcal{K})$. The \textbf{Moore-Penrose inverse} of $\mathcal{T}$ (if it exists) is an element $\mathcal{T}^{\dagger}\in\operatorname{End}_\mathcal{A}^*(\mathcal{H},\mathcal{K})$ satisfying
\begin{itemize}
\item $\mathcal{T} \mathcal{T}^{\dagger} \mathcal{T}=\mathcal{T},$
\item $\mathcal{T}^{\dagger} \mathcal{T} \mathcal{T}^{\dagger}=\mathcal{T}^{\dagger},$
\item $\left(\mathcal{T} \mathcal{T}^{\dagger}\right)^*=\mathcal{T} \mathcal{T}^{\dagger},$
\item $\left(\mathcal{T}^{\dagger} \mathcal{T}\right)^*=\mathcal{T}^{\dagger} \mathcal{T} .$
\end{itemize}
\end{definition}
\begin{lemma} \cite{Fang, Kar}
Let $\Theta \in \operatorname{End}_\mathcal{A}^*(\mathcal{H},\mathcal{K})$. Then the following holds:
\item[1.] $\mathcal{R}(\Theta)$ is closed in $\mathcal{K}$ if and only if $\mathcal{R}\left(\Theta^*\right)$ is closed in $\mathcal{H}$.
\item[2.]  $\left(\Theta^*\right)^{\dagger}=\left(\Theta^{\dagger}\right)^*$.
\item[3.]  The orthogonal projection of $\mathcal{K}$ onto $\mathcal{R}(\Theta)$ is given by $\Theta \Theta^{\dagger}$.
\item[4.]  The orthogonal projection of $\mathcal{H}$ onto $\mathcal{R}\left(\Theta^{\dagger}\right)$ is given by $\Theta^{\dagger} \Theta$.
\end{lemma}
\begin{lemma} \cite{Fang}
 Let $\mathcal{H}_1$ and $\mathcal{H}_2$ two Hilbert $\mathcal{A}$-Modules  and $\mathcal{T}\in \operatorname{End}_\mathcal{A}^*(\mathcal{H}_1,\mathcal{H}) ,\mathcal{T}^{\prime}\in \operatorname{End}_\mathcal{A}^*(\mathcal{H}_2,\mathcal{H}) $ with $\overline{\mathcal{R}\left(\mathcal{T}^{*}\right)}$ is orthogonally
complemented. Then the following assertions are equivalent:
 \begin{enumerate}
 \item $\mathcal{T}^{\prime}\left(\mathcal{T}^{\prime}\right)^* \leq \lambda \mathcal{T} \mathcal{T}^*$ for some $\lambda>0$.
 \item There exist $\mu>0$ such that $\left\|\left(\mathcal{T}^{\prime}\right)^* x\right\| \leq \mu\left\|\mathcal{T}^* x\right\|$ for all $x \in F$.
 \item  There exists $Q \in \operatorname{End}_{\mathcal{A}}^*(\mathcal{H}_2,\mathcal{H}_1 )$ such that $\mathcal{T}^{\prime}=\mathcal{T} Q$, that is the equation $\mathcal{T} X=\mathcal{T}^{\prime}$ has a solution.
 \item $\mathcal{R}\left(\mathcal{T}^{\prime}\right) \subseteq \mathcal{R}(\mathcal{T})$.
 \end{enumerate}
\end{lemma}
\begin{theorem} \cite{Ches}
Let $\left\{\Upsilon_\xi\right\}_{\xi \in \Omega}$ be a c-g-Bessel sequence for $\mathcal{H}$ with respect to $\left\{\mathcal{\mathcal{H_\xi}}\right\}_{\xi \in \Omega}$ with bound $A$. Then the bounded and linear operator $\mathcal{\mathcal{T}}_{\Upsilon} : \bigoplus_{\xi \in \Omega} \mathcal{K}_{\xi}\longrightarrow\mathcal{H}$ weakly defined by
$$
\left\langle \mathcal{T}_{\Upsilon} F, g\right\rangle=\int_{\Omega}\left\langle\Upsilon_\xi^* F(\xi), g\right\rangle \mathrm{d} \nu(\xi), \quad F \in\bigoplus_{\xi \in \Omega} \mathcal{K}_{\xi}, g \in \mathcal{H},
$$ with $\left\|\mathcal{T}_{\Upsilon}\right\| \leq \sqrt{A}$.\\ Moreover, for every  $g \in \mathcal{H}$ and $\xi \in \Omega$,
$$
\mathcal{T}_{\Upsilon}^*(g)(\xi)=\Upsilon_\xi g .
$$
The operators $\mathcal{T}_{\Upsilon}$ and $\mathcal{T}_{\Upsilon}^*$ are called the synthesis and the analysis operator of $\left\{\Upsilon_\xi\right\}_{\xi \in \Omega}$ respectively.
 \end{theorem}
\begin{definition}
Let $K \in \operatorname{End}_\mathcal{A}^*(\mathcal{H})$. A sequence $\left\{\Upsilon_\xi \in \operatorname{End}_\mathcal{A}^*\left(\mathcal{H}, \mathcal{K}_\xi\right): \xi \in\right.\Omega\}$ is called a continuous $K-g$-frame for $\mathcal{H}$ with respect to $\left\{\mathcal{\mathcal{K_\xi}}\right\}_{\xi \in \Omega}$ if:
\begin{enumerate}
\item[a)] for every $h \in \mathcal{H}$, the function $\chi : \Omega \rightarrow \mathcal{K}_\xi$ defined by $\chi(\xi) =\Upsilon_\xi h$ is measurable,
\item[b)] there exist constants $0<A \leq B<\infty$ such that
$$
A \left\langle K^* h,K^* h\right\rangle\leq \int_{\Omega}\left\langle \Upsilon_\xi h,\Upsilon_\xi h\right\rangle \mathrm{d} \nu(\xi) \leq B\left\langle  h, h\right\rangle,
$$
\end{enumerate}
The numbers $A$ and $B$ are called the lower and upper $c-K-g$-frame bounds of $\left\{\Upsilon_\xi \in \operatorname{End}_\mathcal{A}^*\left(\mathcal{H}, \mathcal{K}_\xi\right)\right\}_{\xi \in \Omega}$ , respectively. If $A=B=\delta$, the $c-K-g$-frame is called $\delta$-tight and if $A=B=1$, it is called a Parseval $c-K-g$-frame. If for each $h \in \mathcal{H}$, $$\int_{\Omega}\left\langle \Upsilon_\xi h,\Upsilon_\xi h\right\rangle \mathrm{d} \nu(\xi) \leq B\left\langle  h, h\right\rangle$$ The sequence $\left\lbrace \Upsilon_\xi\right\rbrace_{\xi\in \Omega}$ is called a $c-K-g$-Bessel sequence for $\mathcal{H}$ with respect to $\left\{\mathcal{K_{\xi}}\right\}_{\xi \in \Omega}$.

we suppose that $\left\{\Upsilon_\xi\right\}_{\xi \in \Omega}$ is a c-K-g-frame for $\mathcal{H}$ with respect to $\left\{\mathcal{\mathcal{H_\xi}}\right\}_{\xi \in \Omega}$ with frame bounds $C, D$. The c-K-g-frame operator $S_{\Upsilon}$ : $\mathcal{H} \longrightarrow \mathcal{H}$  is weakly defined by 
$$
\left\langle S_{\Upsilon} f, g\right\rangle=\int_{\Omega}\left\langle \Upsilon_\xi^* \Upsilon_\xi f,  g\right\rangle \mathrm{d} \nu(\xi), \quad f, g \in \mathcal{H}
$$
Furthermore,
$$
C K K^* \leq S_{\Upsilon} \leq D I_{\mathcal{H}}
$$
\end{definition}
\begin{definition} \cite{Aliz2}
A sequence $\left\{\Upsilon_\xi \in \operatorname{End}_\mathcal{A}^*\left(\mathcal{H}, \mathcal{H}_\xi\right): \xi \in\right.\Omega\}$ is c-K-g-frame for $\mathcal{H}$. If for every $g, h \in \mathcal{H}$,
$$
\langle K g, h\rangle=\int_{\Omega}\left\langle\Upsilon_\xi^* \Phi_\xi g, h\right\rangle \mathrm{d} \nu(\xi)
$$
The c-g Bassel sequence $\left\{\Phi_\xi\right\}_{\xi \in \Omega}$ is called a dual c-K-g-Bessel sequence of $\left\{\Upsilon_\xi\right\}_{\xi \in \Omega}$
\end{definition}
\section{Some new constructing of c-K-g-frames in Hilbert $C^{\ast}$-modules }
\begin{theorem}
Let $K \in \operatorname{End}_\mathcal{A}^*(\mathcal{H})$ and $\Theta \in \operatorname{End}_\mathcal{A}^*(\mathcal{H})$,  $\Theta$ has a closed range such that $\Theta K=K \Theta$, suppose that $\left\{\Upsilon_\xi \in \operatorname{End}_\mathcal{A}^*\left(\mathcal{H}, \mathcal{H}_\xi\right)\right\}_{\xi \in \Omega}$ is a c- $K$-g-frame for $\mathcal{H}$, with bounds $A$ and $B$. If $\mathcal{R}\left(K^*\right) \cap \mathcal{N}\left(\Theta^*\right)=\{0\}$ Then for every $f\in \mathcal{H}$,
$$A\left\|(\Theta^*)^{\dagger}\right\|^{-2}\left\langle K^* f,K^* f\right\rangle \leq \int_{\Omega}\left\langle \Upsilon_\xi \Theta^*f,\Upsilon_\xi \Theta^*f\right\rangle \mathrm{d} \nu(\xi) \leq B\|\Theta^*\|^2 \left\langle  f, f\right\rangle
$$
 \end{theorem}
\begin{proof}
suppose that $\left\{\Upsilon_\xi\right\}_{\xi \in \Omega}$ is a c- $K$-g-frame for $\mathcal{H}$ with respect to $\left\{\mathcal{\mathcal{H_\xi}}\right\}_{\xi \in \Omega}$, with bounds $A$ and $B$, 
we have for $f\in \mathcal{H}$
 $$\int_{\Omega}\left\langle \Upsilon_\xi \Theta^*f,\Upsilon_\xi \Theta^*f\right\rangle \mathrm{d} \nu(\xi)\leq B\left\langle  \Theta^*f, \Theta^*f\right\rangle\leq B \|\Theta^*\|^2\left\langle f, f\right\rangle$$
On the other hand, Since $\Theta K=K \Theta$, we have $K^* \Theta^*=\Theta^* K^*$. We assume that$\mathcal{R}\left(K^*\right) \cap \mathcal{N}\left(\Theta^*\right)=\{0\}$ and $\Theta$ has closed range, using Lemma \textbf{1.7} for every $f \in \mathcal{H}$,
$$
\begin{aligned}
\left\langle K^* f,K^* f\right\rangle & =\left\langle \Theta \Theta^{\dagger} K^* f, \Theta \Theta^{\dagger} K^* f\right\rangle \\&=\left\langle\left(\Theta^{\dagger}\right)^* \Theta^* K^* f,\left(\Theta^{\dagger}\right)^* \Theta^* K^* f\right\rangle \\&=\left\langle\left(\Theta^*\right)^{\dagger}  K^* \Theta^* f,\left(\Theta^*\right)^{\dagger} K^* \Theta^* f\right\rangle \\
& \leq\left\|(\Theta^*)^{\dagger}\right\|^2\left\langle K^* \Theta^* f,K^* \Theta^* f\right\rangle
\end{aligned}
$$
Hence for each $f \in \mathcal{H}$,
$$
\int_{\Omega}\left\langle \Upsilon_\xi \Theta^* f,\Upsilon_\xi \Theta^* f\right\rangle \mathrm{d} \nu(\xi) \geq A\left\langle K^* \Theta^* f, K^* \Theta^* f\right\rangle \geq A\left\|\left(\Theta^*\right)^{\dagger} \right\|^{-2}\left\langle K^* f,K^* f\right\rangle .
$$
\end{proof}
\begin{corollary}
Let $K \in \operatorname{End}_\mathcal{A}^*(\mathcal{H})$ such that $\overline{\mathcal{R}(K)}=\mathcal{H}$, $\Theta \in \operatorname{End}_\mathcal{A}^*(\mathcal{H})$ has closed range and $\Theta K=K \Theta$. If $\left\{\Upsilon_\xi \Theta\right\}_{\xi \in \Omega}$ and $\left\{\Upsilon_\xi \Theta^*\right\}_{\xi \in \Omega}$ are both $c$ - $K$-g-frames for $\mathcal{H}$ with respect to $\left\{\mathcal{\mathcal{H_\xi}}\right\}_{\xi \in \Omega}$, then $\left\{\Upsilon_\xi\right\}_{\xi \in \Omega}$ is a $c-K$-g-frame for $\mathcal{H}$ with respect to $\left\{\mathcal{\mathcal{H_\xi}}\right\}_{\xi \in \Omega}$.
 \end{corollary}
\begin{proof}
As such $\overline{\mathcal{R}(K)}=\mathcal{H}$, thus  $\mathcal{N}\left(K^*\right)=\{0\}$ and $\mathcal{N}\left(K^*\right)^{\perp}=\mathcal{H}$ . For every $f \in \mathcal{H}$, we have
$$
A\left\langle K^* f,K^* f \right\rangle \leq \int_{\Omega}\left\langle \Upsilon_\xi \Theta^* f,\Upsilon_\xi \Theta^* f \right\rangle \mathrm{d} \nu(\xi),
$$
Hence $\mathcal{N}\left(K^*\right) \supseteq \mathcal{N}\left(\Theta^*\right)$, so
$$
\mathcal{H}=\mathcal{N}\left(K^*\right)^{\perp} \subseteq \mathcal{N}\left(\Theta^*\right)^{\perp}=\mathcal{R}(\Theta) .
$$
consequently $\Theta$ is surjective. Also, for every $f \in \mathcal{H}$,
$$
A\left\langle K^* f,K^* f \right\rangle \leq \int_{\Omega}\left\langle \Upsilon_\xi \Theta f,\Upsilon_\xi \Theta f \right\rangle \mathrm{d} \nu(\xi),
$$
thus $ \mathcal{N}(\Theta) \subseteq \mathcal{N}\left(K^*\right)=\{0\}$. Therefore $\Theta$ is invertible. Since $\Theta K=K \Theta$, we have $ \Theta^{-1} K=K \Theta^{-1}, \mathcal{R}\left(K^*\right) \cap \mathcal{N}\left(\left(\Theta^{-1}\right)^*\right)=\{0\}$, and $$\left\{\Upsilon_\xi :\xi \in \Omega \right\}=\left\{\Upsilon_\xi\left( \Theta^{-1}\Theta\right)^* :\xi \in \Omega\right\}=\left\{\left(\Upsilon_\xi \Theta^*\right)\left(\Theta^{-1}\right)^* :\xi \in \Omega \right\}.$$ Hence, according to Theorem 2.1, We conclude that $\left\{\Upsilon_\xi\right\}_{\xi \in \Omega}$ is a c-K-g-frame for $\mathcal{H}$.
\end{proof}
\begin{theorem}
Let $\Theta,K \in \operatorname{End}_\mathcal{A}^*(\mathcal{H})$ and  $\left\{\Upsilon_\xi \in \operatorname{End}_\mathcal{A}^*\left(\mathcal{H}, \mathcal{H}_\xi\right): \xi \in\right.\Omega\}$ be a $\delta$-tight c- $K$ $g$-frame for $\mathcal{H}$. If $\Theta K=K \Theta$ and $K^*$ is bounded below. Then the following assertions are equivalent:
\begin{enumerate}
\item[(i)]$\Theta$ is surjective.
\item[(ii)]$\left\{\Upsilon_\xi \Theta^*\right\}_{\xi \in \Omega}$ is a c-K-g-frame for $\mathcal{H}$.
\end{enumerate}
 \end{theorem}
\begin{proof}
$(i)\Rightarrow(ii)$ The first part of proof is implied by Theorem $2.1$. .\\
$(ii)\Rightarrow(i)$ Suppose that for all $f \in \mathcal{H}$
$$
A\left\langle K^* f,K^* f \right\rangle \leq \int_{\Omega}\left\langle \Upsilon_\xi \Theta^* f,\Upsilon_\xi \Theta^* f \right\rangle \mathrm{d} \nu(\xi) \leq B\left\langle  f, f\right\rangle .
$$
Moreover for every $h \in \mathcal{H}$,
$$
\delta\left\langle K^* h,K^* h\right\rangle=\int_{\Omega}\left\langle\Upsilon_\xi h,\Upsilon_\xi h\right\rangle \mathrm{d} \nu(\xi) .
$$
since $\Theta^* K^*=K^* \Theta^*$, we obtain
$$
\delta\left\langle \Theta^* K^* f,\Theta^* K^* f\right\rangle=\int_{\Omega}\left\langle \Upsilon_\xi \Theta^* f,\Upsilon_\xi \Theta^* f \right\rangle \mathrm{d} \nu(\xi), \quad f \in \mathcal{H} .
$$
Hence
$$
\left\langle \Theta^* K^* f,\Theta^* K^* f\right\rangle=\delta^{-1} \int_{\Omega}\left\langle \Upsilon_\xi \Theta^* f,\Upsilon_\xi \Theta^* f \right\rangle \mathrm{d} \nu(\xi) \geq \delta^{-1} A\left\langle K^* f,K^* f \right\rangle .
$$
Since $K^*$ is bounded below, by Lemma \textbf{1.5} there exist $\alpha>0$ such that $$\left\langle K^* f,K^* f\right\rangle \geq\alpha\left\langle f ,f \right\rangle.$$ So, we conclude that for every $f \in \mathcal{H},$
$$
\left\langle \Theta^* K^* f,\Theta^* K^* f\right\rangle \geq \delta^{-1} \alpha A\left\langle f ,f \right\rangle .
$$
Consequently $(K \Theta )^*=\Theta^* K^*$ is bounded below, so by Lemma 1.5, $ K \Theta$ is surjective and since $K$ and $ \Theta$ commute, implies that $\Theta$ is surjective.
\end{proof}
Assume that operators $\mathcal{T}, \Theta \in \operatorname{End}_\mathcal{A}^*(\mathcal{H})$ and $\mathcal{T}^*$ conserves a c-K-g-frame for $\mathcal{R}(\mathcal{T})$. We establish some requirements on $K, \Theta$ and $\mathcal{T}$ in the preceding theorem such that $\Theta^*$ can also conserve the same c-K-g-frame for $\mathcal{R}(\Theta)$.
\begin{theorem}
Let $\left\{\Upsilon_\xi \in \operatorname{End}_\mathcal{A}^*\left(\mathcal{H}, \mathcal{H}_\xi\right): \xi \in\right.\Omega\}$ be a c-K-g-frame for $\mathcal{H}$. Assume that $ \Theta, \mathcal{T} \in \operatorname{End}_\mathcal{A}^*(\mathcal{H})$ have close ranges, and $\mathcal{N}(\Theta)=\mathcal{N}(\mathcal{T})$ with $\mathcal{R}\left(K^*\right) \cap \mathcal{N}\left(\Theta^*\right)=\{0\}$ and $K \Theta \mathcal{T}^{\dagger}=\Theta \mathcal{T}^{\dagger} K$. If $\left\{\Upsilon_\xi \mathcal{T}^*\right\}_{\xi \in \Omega}$ is a c-K-g-frame for $\mathcal{R}(\mathcal{T})$, then $\left\{\Upsilon_\xi \Theta^*\right\}_{\xi \in \Omega}$ is a c- $K$-g-frame for $\mathcal{R}(\Theta)$.

 \end{theorem}
\begin{proof} Assume that $K \Theta \mathcal{T}^{\dagger}=\Theta \mathcal{T}^{\dagger} K$.  We define
$$
P :\mathcal{R}(\mathcal{T}) \longrightarrow \mathcal{R}(\Theta),
$$
by $P f=\Theta \mathcal{T}^{\dagger} f$, for all $ f \in \mathcal{R}(\mathcal{T})$. Since $\mathcal{N}(\mathcal{T})=\mathcal{N}(\Theta)$, we have $\mathcal{R}\left(\mathcal{T}^{\dagger}\right)=\mathcal{R}\left(\Theta^{\dagger}\right)$. Consequently by Lemma \textbf{1.7}, $\mathcal{N}(P)=\mathcal{N}\left(\Theta \mathcal{T}^{\dagger}\right)=\mathcal{N}\left(\mathcal{T} \mathcal{T}^{\dagger}\right)=(\mathcal{R}(\mathcal{T}))^{\perp}$, that implies 
$$
\mathcal{N}(P)=\mathcal{N}\left(\Theta \mathcal{T}^{\dagger}\right) \cap \mathcal{R}(\mathcal{T})=(\mathcal{R}(\mathcal{T}))^{\perp} \cap \mathcal{R}(\mathcal{T})=\{0\} .
$$
Hence, $P$ is invertible on $\mathcal{R}(\mathcal{T})$. By Lemma \textbf{1.7},$$ \mathcal{T}^{\dagger} \mathcal{T}=P_{\mathcal{R}\left(\mathcal{T}^{\dagger}\right)}=P_{\mathcal{R}\left(\Theta^{\dagger}\right)}=\Theta^{\dagger} \Theta .$$ Furthermore
\begin{equation}
P \mathcal{T}=\Theta \mathcal{T}^{\dagger} \mathcal{T}=\Theta \Theta^{\dagger} \Theta=\Theta.
\end{equation}
Let $C$, 
 be the lawer frame bound of $\left\{\Upsilon_\xi \mathcal{T}^*\right\}_{\xi \in \Omega}$, then for every $f \in$ $\mathcal{R}(\Theta)$, by (2.1), we obtain $$
\begin{aligned}
\int_{\Omega}\left\langle \Upsilon_\xi \Theta^* f,\Upsilon_\xi \Theta^* f \right\rangle \mathrm{d} \nu(\xi) & =\int_{\Omega}\left\langle\Upsilon_\xi \mathcal{T}^* P^* f,\Upsilon_\xi \mathcal{T}^* P^* f\right\rangle \mathrm{d} \nu(\xi)\\ & \geq C\left\langle K^* P^* f,K^* P^* f\right\rangle \\
& \geq C\left\| (P^*)^{-1}\right\|^{-2}\left\langle K^* f,K^* f \right\rangle .
\end{aligned}
$$
On the other hand, for every $f \in \mathcal{R}(\Theta)$,
$$
\int_{\Omega}\left\langle \Upsilon_\xi \Theta^* f,\Upsilon_\xi \Theta^* f \right\rangle \mathrm{d} \nu(\xi)=\int_{\Omega}\left\langle \Upsilon_\xi \mathcal{T}^*P^* f,\Upsilon_\xi \mathcal{T}^*P^* f \right\rangle \mathrm{d} \nu(\xi) \leq B\left\langle P^* f,P^* f \right\rangle =B\|P^*\|^2\left\langle  f, f\right\rangle .
$$
Which implies that $\left\{\Upsilon_\xi \Theta^*\right\}_{\xi \in \Omega}$ is a c-K-g-frame for $\mathcal{R}(\Theta)$.
\end{proof}
\begin{theorem}
Let $K \in \operatorname{End}_\mathcal{A}^*(\mathcal{H})$ and $\left\{\Upsilon_\xi \in \operatorname{End}_\mathcal{A}^*\left(\mathcal{H}, \mathcal{H}_\xi\right): \xi \in\right.\Omega\}$ be a c-g-Bessel sequence for $\mathcal{H}$, and $\mathcal{T}_{\Upsilon}$ be the synthesis operator of $\left\{\Upsilon_\xi \right\}_{\xi \in \Omega}$ with $\overline{\mathcal{R}\left(K^{*}\right)}$ is orthogonally
complemented. Then the following assertions are equivalent:
\begin{enumerate}
\item[(1)] $\mathcal{R}(K)=\mathcal{R}\left(\mathcal{T}_{\Upsilon}\right)$.
\item[(2)] There exist two constants $\lambda_1, \lambda_2>0$, such that for each $f \in \mathcal{H}$,
\begin{equation}
\frac{1}{\lambda_1}\left\langle K^* f,K^* f \right\rangle \leq \int_{\Omega}\left\langle \Upsilon_\xi f,\Upsilon_\xi f\right\rangle \mathrm{d} \nu(\xi) \leq \lambda_2\left\langle K^* f,K^* f \right\rangle .
\end{equation}
\item[(3)]$\left\{\Upsilon_\xi\right\}_{\xi \in \Omega}$ is a c-K-g-frame for $\mathcal{H}$ with respect to $\left\{\mathcal{\mathcal{H_\xi}}\right\}_{\xi \in \Omega}$ and there exists a c-g-Bessel sequence $\left\{\Phi_\xi\right\}_{\xi \in \Omega}$ for $\mathcal{H}$ with respect to $\left\{\mathcal{\mathcal{H_\xi}}\right\}_{\xi \in \Omega}$ such that $\Upsilon_\xi=\Phi_\xi K^*$ for each $\xi \in \Omega$.
\end{enumerate}

 \end{theorem}
\begin{proof}
(1) Applying Lemma $1.8$, $$K K^* \leq \lambda_1\mathcal{T}_{\Upsilon} \mathcal{T}_{\Upsilon}^*, \quad and \quad \mathcal{T}_{\Upsilon} \mathcal{T}_{\Upsilon}^* \leq \lambda_2 K K^* \quad for \; some \; \lambda_1,\lambda_2>0. $$ Which implies that $$\frac{1}{\lambda_1} K K^* \leq \mathcal{T}_{\Upsilon} \mathcal{T}_{\Upsilon}^* \leq \lambda_2 K K^*.$$ Hence, for every $f \in \mathcal{H}$,
$$
\frac{1}{\lambda_1}\left\langle K^* f,K^* f \right\rangle \leq\left\langle \mathcal{T}_{\Upsilon}^* f,\mathcal{T}_{\Upsilon}^* f\right\rangle=\int_{\Omega}\left\langle \Upsilon_\xi f,\Upsilon_\xi f\right\rangle \mathrm{d} \nu(\xi) \leq \lambda_2 \left\langle K^* f,K^* f \right\rangle
$$
(2) $\Rightarrow$ (3) According to the assumed hypothesis we have $\mathcal{T}_{\Upsilon} \mathcal{T}_{\Upsilon}^* \leq \lambda_2 K K^*$. Applying Lemma \textbf{1.8}, there exists an operator $\mathcal{Q} \in \operatorname{End}_\mathcal{A}^*\left(\bigoplus_{\xi \in \Omega} \mathcal{K}_{\xi}, \mathcal{H}\right)$ such that $\mathcal{T}_{\Upsilon}=K \mathcal{Q}$, Which implies that $\mathcal{T}_{\Upsilon}^*=\mathcal{Q}^* K^*$. Now for any $h \in \mathcal{H}$ and for almost all $\xi \in \Omega$, we define:
$$
\Phi_\xi h=\left(\mathcal{Q}^* h\right)(\xi) .
$$
Consequently,
$$
\left\{\Upsilon_\xi(h)\right\}_{\xi \in \Omega}=\left\{\mathcal{T}_{\Upsilon}^* h\right\}_{\xi \in \Omega}=\left\{\left(\mathcal{Q}^*\left(K^* h\right)(\xi)\right\}_{\xi \in \Omega}=\left\{\Phi_\xi\left(K^* h\right)\right\}_{\xi \in \Omega},\right.
$$
Hence $\Upsilon_\xi=\Phi_\xi K^*$ for almost all $\xi \in \Omega$. Hence for every $h \in \mathcal{H}$, we achieve the intended result by:
$$
\int_{\Omega}\left\langle\Phi_\xi h,\Phi_\xi h \right\rangle \mathrm{d} \nu(\xi)=\int_{\Omega}\left\langle \left(\mathcal{Q}^* h\right)(\xi),\left(\mathcal{Q}^* h\right)(\xi)\right\rangle \mathrm{d} \nu(\xi) \leq\|\mathcal{Q}^*\|_2^2\left\langle h,h \right\rangle  .
$$
(3) $\Rightarrow$ (1) For every $f \in \mathcal{H}$,
$$
\frac{1}{\lambda_1}\left\langle K^* f,K^* f \right\rangle \leq \int_{\Omega}\left\langle \Upsilon_\xi f,\Upsilon_\xi f\right\rangle d \nu\left(\xi\right)=\int_{\Omega}\left\langle\Phi_\xi K^* f,\Phi_\xi K^* f\right\rangle \mathrm{d} \nu(\xi) \leq \lambda_{\Phi}\left\langle K^* f,K^* f \right\rangle .
$$
where $\lambda_{\Phi}$ the upper bound of $\left\{\Phi_\xi \right\}_{\xi \in \Omega}$.
 Hence $$\frac{1}{\lambda_1} K K^* \leq \mathcal{T}_{\Upsilon} \mathcal{T}_{\Upsilon}^* \leq \lambda_{\Phi} K K^*.$$ Hence $$\mathcal{R}(K)=\mathcal{R}\left(\mathcal{T}_{\Upsilon}\right).$$

\end{proof}
Given certain adjointable operators and some c-K-g-frames, the following theorem is used to construct c-K-g-frames in Hilbert $C^*$ modules. 
\begin{theorem}
Let $\left\{\Upsilon_\xi \in \operatorname{End}_\mathcal{A}^*\left(\mathcal{H}, \mathcal{H}_\xi\right): \xi \in\right.\Omega\}$ be a $c - K_1-g$ frame for $\mathcal{H}$ and $K_1, K_2 \in \operatorname{End}_\mathcal{A}^*(\mathcal{H})$ with $\overline{\mathcal{R}\left(K_{1}^{*}\right)}$ is orthogonally
complemented.
\begin{enumerate}
\item[(1)]If $\left\{\Upsilon_\xi \in \operatorname{End}_\mathcal{A}^*\left(\mathcal{H}, \mathcal{H}_\xi\right): \xi \in\right.\Omega\}$ is a $c- K_2-g$ frame for $\mathcal{H}$, then it is a $c - \left(K_1+K_2\right)-g$ frame for $\mathcal{H}$.
\item[(2)]If, in addition, $\left\{\Upsilon_\xi \in \operatorname{End}_\mathcal{A}^*\left(\mathcal{H}, \mathcal{H}_\xi\right): \xi \in\right.\Omega\}$ is $\delta$-tight $c-K_1-g$ frame, then it is a c- $K_2-g$-frame for $\mathcal{H}$ if and only if $\mathcal{R}\left(K_2\right) \subseteq \mathcal{R}\left(K_1\right)$.
\end{enumerate}
 \end{theorem}
\begin{proof}
 (1) Assume that $\left\{\Upsilon_\xi\right\}_{\xi \in \Omega}$ is a $c$ - $K_1-g$-frame and also $c-K_2-g$ frame for $\mathcal{H}$ with respect to $\left\{\mathcal{\mathcal{H_\xi}}\right\}_{\xi \in \Omega}$, then there exist $A_1,A_2,B_1,B_2>0$ constants such that for every $f \in \mathcal{H}$, 
$$
A_1\left\langle K_1^* f,K_1^* f \right\rangle \leq \int_{\Omega}\left\langle \Upsilon_\xi f,\Upsilon_\xi f\right\rangle \mathrm{d} \nu(\xi)\leq B_1\left\langle  f,f\right\rangle.
$$
And
$$
A_2\left\langle K_2^* f,K_2^* f \right\rangle \leq \int_{\Omega}\left\langle \Upsilon_\xi f,\Upsilon_\xi f\right\rangle \mathrm{d} \nu(\xi)\leq B_2\left\langle  f,f\right\rangle.
$$
It follow that
$$
\left(\frac{A_1}{2}\left\langle K_1^* f,K_1^* f \right\rangle+\frac{A_1}{2}\left\langle K_2^* f,K_2^* f \right\rangle\right) \leq \int_{\Omega}\left\langle \Upsilon_\xi f,\Upsilon_\xi f\right\rangle \mathrm{d} \nu(\xi)\leq \left(\frac{B_1}{2}\left\langle  f,f\right\rangle + \frac{B_2}{2}\left\langle  f,f\right\rangle \right).
$$
we pose $\lambda_1=\min \left\{\frac{A_1}{2}, \frac{A_2}{2}\right\}$ and $\lambda_2=\max \left\{\frac{B_1}{2}, \frac{B_2}{2}\right\}$, since
$$
\begin{aligned}
\left\|K_1^* f\right\|^2 & =\left\|K_1^* f\right\|^2=\left\|\left( K_1^*+ K_2^*\right) f- K_2^* f\right\|^2=\left\|\left( K_1+ K_2\right)^* f- K_2^* f\right\|^2 \\
& \geq \left\|\left( K_1+ K_2\right)^* f\right\|^2-\left\| K_2^* f\right\|^2.
\end{aligned}
$$
We have
$$
\left\|\left( K_1+ K_2\right)^* f\right\|^2 \leq \left\|K_1^* f\right\|^2+\left\|K_2^* f\right\|^2.
$$
Hence
$$
\lambda_1\left\langle\left(K_1+K_2\right)^* f,\left(K_1+K_2\right)^* f\right\rangle \leq \int_{\Omega}\left\langle \Upsilon_\xi f,\Upsilon_\xi f\right\rangle \mathrm{d} \nu(\xi)\leq \lambda_2 \left\langle  f,f\right\rangle.
$$ 
(2) For every $f \in \mathcal{H}$,
 \begin{equation}
 \delta\left\langle K_1^* f,K_1^* f \right\rangle=\int_{\Omega}\left\langle \Upsilon_\xi f,\Upsilon_\xi f\right\rangle \mathrm{d} \nu(\xi).
 \end{equation}
On the other hand, we have $\left\{\Upsilon_\xi \in \operatorname{End}_\mathcal{A}^*\left(\mathcal{H}, \mathcal{H}_\xi\right): \xi \in\right.\Omega\}$ is $c-K_2-g$-frame, then there exists a $A>0$, such that 
\begin{equation}
\int_{\Omega}\left\langle \Upsilon_\xi f,\Upsilon_\xi f\right\rangle \mathrm{d} \nu(\xi) \geq A\left\langle K_2^* f,K_2^* f \right\rangle 
\end{equation}
Hence, from (2.3) and (2.4), $$K_2 K_2^* \leq \frac{\delta}{A} K_1 K_1^*.$$ Applying Lemma \textbf{1.8} we get $\mathcal{R}\left(K_2\right) \subseteq \mathcal{R}\left(K_1\right)$.\\ Now
to show the opposite inclusion, simply use the lemma \textbf{1.8}, there exists $\gamma>0$, such that $K_2 K_2^* \leq \gamma K_1 K_1^*$. Hence for each $f \in \mathcal{H}$, we have
$$
\left\langle K_2^* f,K_2^* f \right\rangle \leq \gamma\left\langle K_1^* f,K_1^* f \right\rangle=\frac{\gamma}{\delta} \int_{\Omega}\left\langle \Upsilon_\xi f,\Upsilon_\xi f\right\rangle \mathrm{d} \nu(\xi),
$$
Therefore
$$
\frac{\delta}{\gamma}\left\langle K_2^* f,K_2^* f \right\rangle \leq \int_{\Omega}\left\langle \Upsilon_\xi f,\Upsilon_\xi f\right\rangle \mathrm{d} \nu(\xi) .
$$
this completes the proof
\end{proof}
\section{Sum of c-K-g-frames in Hilbert $C^*$ modules}
In this part, we investigate the sum of these frames under the assumption that $\left\{\Upsilon_\xi\right\}_{\xi \in \Omega}$ and $\left\{\Phi_\xi\right\}_{\xi \in \Omega}$ are arbitrary.
\begin{theorem}
Let $K_1, K_2 \in \operatorname{End}_\mathcal{A}^*(\mathcal{H})$ have closed ranges, $\left\{\Upsilon_\xi\right\}_{\xi \in \Omega}$ and $\left\{\Phi_\xi\right\}_{\xi \in \Omega}$ are $c- K_1-g$ frame and $c-g$ Bessel sequence for $\mathcal{H}$ with respect to $\left\{\mathcal{\mathcal{H_\xi}}\right\}_{\xi \in \Omega}$, respectively.\begin{enumerate}
\item[(i)] If $K_1 \geq 0$ and $\left\{\Phi_\xi\right\}_{\xi \in \Omega}$ is a $c- K_{1}-g$ dual for $\left\{\Upsilon_\xi\right\}_{\xi \in \Omega}$, then the sequence $\left\{\Upsilon_\xi+\Phi_\xi\right\}_{\xi \in \Omega}$ is a c- $K_1-g$-frame for $\mathcal{H}$ with respect to $\left\{\mathcal{\mathcal{H_\xi}}\right\}_{\xi \in \Omega}$.
\item[(ii)] If $\left\{\Phi_\xi\right\}_{\xi \in \Omega}$ is c-$\left(K_1+K_2\right)-g$-frame for $\mathcal{H}$ with respect to $\left\{\mathcal{\mathcal{H_\xi}}\right\}_{\xi \in \Omega}$ and $\mathcal{T}_{\Upsilon} \mathcal{T}_{\Phi}^*=0$, then $\left\{\Upsilon_\xi+\Phi_\xi\right\}_{\xi \in \Omega}$ is a c- $\left(K_1+K_2\right)-g$-frame for $\mathcal{H}$ with respect to $\left\{\mathcal{\mathcal{H_\xi}}\right\}_{\xi \in \Omega}$.
\end{enumerate}
 \end{theorem}
\begin{proof}
(i)  For every $f \in \mathcal{H}$, we have
$$
\begin{aligned}
\left\langle K_1^* f, h\right\rangle & =\left\langle f, K_1 h\right\rangle\\ &=\overline{\left\langle K_1 h, f\right\rangle}
\end{aligned}
$$
Since $\left\{\Phi_\xi\right\}_{\xi \in \Omega}$ is a $c-K_1-g$-dual of $\left\{\Upsilon_\xi\right\}_{\xi \in \Omega}$ 
$$
\begin{aligned}
\overline{\left\langle K_1 h, f\right\rangle}&=\overline{\int_{\Omega}\left\langle\Upsilon_\xi^* \Phi_\xi h, f\right\rangle \mathrm{d} \nu(\xi)} \\
& =\int_{\Omega}\left\langle\Phi_\xi^* \Upsilon_\xi f, h\right\rangle \mathrm{d} \nu(\xi) .
\end{aligned}
$$
We denote by $S_{\Upsilon+\Phi}$, the $c-g$ frame operator of $\left\{\Upsilon_\xi+\Phi_\xi\right\}_{\xi \in \Omega}$. Consequently for every $f, h \in \mathcal{H}$,
$$
\begin{aligned}
\left\langle S_{\Upsilon+\Phi} f, h\right\rangle & =\int_{\Omega}\left\langle f,\left(\Upsilon_\xi+\Phi_\xi\right)^*\left(\Upsilon_\xi+\Phi_\xi\right) h\right\rangle \mathrm{d} \nu(\xi) \\
& =\int_{\Omega}\left\langle\left(\Upsilon_\xi+\Phi_\xi\right)^*\left(\Upsilon_\xi+\Phi_\xi\right) f, h\right\rangle \mathrm{d} \nu(\xi) \\
& =\int_{\Omega}\left\langle\Upsilon_\xi^* \Upsilon_\xi f, h\right\rangle \mathrm{d} \nu(\xi)+\int_{\Omega}\left\langle\Phi_\xi^* \Phi_\xi f, h\right\rangle \mathrm{d} \nu(\xi) \\
& +\int_{\Omega}\left\langle\Upsilon_\xi^* \Phi_\xi f, h\right\rangle \mathrm{d} \nu(\xi)+\int_{\Omega}\left\langle\Phi_\xi^* \Upsilon_\xi f, h\right\rangle \mathrm{d} \nu(\xi) \\
& =\left\langle S_{\Upsilon} f, h\right\rangle+\left\langle S_{\Phi} f, h\right\rangle+\left\langle K_1 f, h\right\rangle+\left\langle K_1^* f, h\right\rangle.
\end{aligned}
$$
Therefore
$$
\begin{aligned}
\left\langle S_{\Upsilon+\Phi} f, f\right\rangle & =\int_{\Omega}\left\langle\left(\Upsilon_\xi+\Phi_\xi\right)^*\left(\Upsilon_\xi+\Phi_\xi\right) f, f\right\rangle \mathrm{d} \nu(\xi) \\ &\geq\left\langle S_{\Upsilon} f, f\right\rangle \\
& =\int_{\Omega}\left\langle \Upsilon_\xi f,\Upsilon_\xi f\right\rangle \mathrm{d} \nu(\xi) \\ &\geq C_{\Upsilon}\left\langle K_1^* f,K_1^* f \right\rangle .
\end{aligned}
$$
This prove that $\left\{\Upsilon_\xi+\Phi_\xi\right\}_{\xi \in \Omega}$ has the lower frame bound.
Now, we prove $\left\{\Upsilon_\xi+\Phi_\xi\right\}_{\xi \in \Omega}$ is a $c$-g-Bessel sequence. For every $f \in \mathcal{H}$,
$$
\begin{aligned}
\int_{\Omega}\left\langle\left(\Upsilon_\xi+\Phi_\xi\right) f,\left(\Upsilon_\xi+\Phi_\xi\right) f\right\rangle  \mathrm{d} \nu(\xi) & \leq  \int_{\Omega}\left\langle \Upsilon_\xi f,\Upsilon_\xi f\right\rangle \mathrm{d} \nu(\xi)+ \int_{\Omega}\left\langle \Phi_\xi f,\Phi_\xi f \right\rangle \mathrm{d} \nu(\xi) \\
& \leq  B_1\left\langle  f, f\right\rangle+ B_2\left\langle  f, f\right\rangle=\left(B_1+B_2\right)\left\langle  f, f\right\rangle .
\end{aligned}
$$(ii)  For every $f \in \mathcal{H}$, Since $\mathcal{T}_{\Upsilon} \mathcal{T}_{\Phi}^*=0$, we have $\int_{\Omega}\left\langle\Lambda_\xi^* \Phi_\xi f, f\right\rangle \mathrm{d} \nu(\xi)=0$ and
$$
\begin{aligned}
\int_{\Omega}\left\langle \left(\Upsilon_\xi+\Phi_\xi\right) f,\left(\Upsilon_\xi+\Phi_\xi\right) f \right\rangle \mathrm{d} \nu(\xi) & =\int_{\Omega}\left\langle \Upsilon_\xi f,\Upsilon_\xi f\right\rangle \mathrm{d} \nu(\xi)+\int_{\Omega}\left\langle \Phi_\xi f,\Phi_\xi f \right\rangle \mathrm{d} \nu(\xi) \\
& \geq A_1\left\langle K_1^* f,K_1^* f \right\rangle+A_2\left\langle K_2^* f,K_2^* f \right\rangle\\ &\geq \lambda\left\langle\left(K_1+K_2\right)^* f,\left(K_1+K_2\right)^* f\right\rangle
\end{aligned}
$$
where $\lambda=\min \left\{A_1, A_2\right\}$.   
\end{proof}
\begin{theorem}
Let $K_1 \in \operatorname{End}_\mathcal{A}^*(\mathcal{H}), K_2 \in \operatorname{End}_\mathcal{A}^*(\mathcal{K})$ and $\Upsilon=$ $\left\{\Upsilon_\xi\right\}_{\xi \in \Omega}$ is a c-K $K_1$-g-frame and $\left\{\Phi_\xi\right\}_{\xi \in \Omega}$ is a c-g-Bessel sequence for $\mathcal{H}$. Assume that $\Theta_1, \Theta_2 \in \operatorname{End}_\mathcal{A}^*(\mathcal{H},\mathcal{K})$ and $\Theta_1 \mathcal{T}_{\Upsilon} \mathcal{T}_{\Phi}^* \Theta_2^*+\Theta_2 \mathcal{T}_{\Phi} \mathcal{T}_{\Upsilon}^* \Theta_1^*+$ $\Theta_2 S_{\Phi} \Theta_2^* \geq 0$. If $\Theta_1$ has closed range with $\Theta_1 K_1=K_2 \Theta_1$ and $\mathcal{R}\left(K_2^*\right) \cap$ $\mathcal{N}\left(\Theta_1^*\right)=\{0\}$, then $\left\{\Upsilon_\xi \Theta_1^*+\Phi_\xi \Theta_2^*\right\}_{\xi \in \Omega}$ is a c-$K_{2}$-g-frame for $\mathcal{K}$ with respect to $\left\{\mathcal{\mathcal{H_\xi}}\right\}_{\xi \in \Omega}$.

 \end{theorem}
\begin{proof}
Assume that $\left\{\Upsilon_\xi\right\}_{\xi \in \Omega}, \left\{\Phi_\xi\right\}_{\xi \in \Omega}$ be a $c-K_1-g$-frame and $c-g$ Bessel sequence for $\mathcal{H}$ with bounds $\alpha_1, \beta_1$ and $\beta_2$, respectively. For every $g \in \mathcal{K}$,
$$
\begin{aligned}
\int_{\Omega}\left\langle \left(\Upsilon_\xi \Theta_1^*+\Phi_\xi \Theta_2^*\right) g,\left(\Upsilon_\xi \Theta_1^*+\Phi_\xi \Theta_2^*\right) g \right\rangle \mathrm{d} \nu(\xi) & =\int_{\Omega}\left\langle \Upsilon_\xi \Theta_1^* g,\Upsilon_\xi \Theta_1^* g \right\rangle \mathrm{d} \nu(\xi)+\left\langle \Theta_2 \mathcal{T}_{\Phi} \mathcal{T}_{\Upsilon}^* \Theta_1^* g, g\right\rangle \\
& +\left\langle \Theta_1 \mathcal{T}_{\Upsilon} \mathcal{T}_{\Phi}^* \Theta_2^* g, g\right\rangle+\left\langle \Theta_2 \mathcal{T}_{\Phi} \mathcal{T}_{\Phi}^* \Theta_2^* g, g\right\rangle \\
& =\int_{\Omega}\left\langle \Upsilon_\xi \Theta_1^* g,\Upsilon_\xi \Theta_1^* g \right\rangle \mathrm{d} \nu(\xi)+\left\langle\left( \Theta_1 \mathcal{T}_{\Upsilon} \mathcal{T}_{\Phi}^* \Theta_2^*\right.\right. \\
& \left.\left.+\Theta_2 \mathcal{T}_{\Phi} \mathcal{T}_{\Upsilon}^* \Theta_1^*+\Theta_2 S_{\Phi} \Theta_2^*\right) g, g\right\rangle
\end{aligned}
$$
According to the hypotheses, for every  $g \in \mathcal{H}$ we get
$$
\begin{aligned}
\int_{\Omega}\left\langle \left(\Upsilon_\xi \Theta_1^*+\Phi_\xi \Theta_2^*\right) g,\left(\Upsilon_\xi \Theta_1^*+\Phi_\xi \Theta_2^*\right) g \right\rangle \mathrm{d} \nu(\xi) & \geq \int_{\Omega}\left\langle \Upsilon_\xi \Theta_1^* g,\Upsilon_\xi \Theta_1^* g \right\rangle \mathrm{d} \nu(\xi)\\  &\geq \alpha_1\left\langle K_1^* \Theta_1^* g,K_1^* \Theta_1^* g\right\rangle \\
& =\alpha_1\left\langle \Theta_1^* K_2^* g,\Theta_1^* K_2^* g\right\rangle \\  & \geq \alpha_1\left\|\Theta_1^{\dagger}\right\|^{-2}\left\langle K_2^* g,K_2^* g\right\rangle
\end{aligned}
$$Hence , for every $g \in \mathcal{K}$,
$$
\begin{aligned}
\alpha_1\left\|\Theta_1^{\dagger}\right\|^{-2}\left\langle K_2^* g,K_2^* g\right\rangle & \leq \int_{\Omega}\left\langle \left(\Upsilon_\xi \Theta_1^*+\Phi_\xi \Theta_2^*\right) g,\left(\Upsilon_\xi \Theta_1^*+\Phi_\xi \Theta_2^*\right) g \right\rangle \mathrm{d} \nu(\xi) \\
& \leq\left( \beta_1\left\|\Theta_1^{*}\right\|^2+ \beta_2\left\|\Theta_2^{*}\right\|^2\right)\left\langle g, g\right\rangle .
\end{aligned}
$$
\end{proof}

\begin{theorem}
Let $K_1 \in \operatorname{End}_\mathcal{A}^*(\mathcal{H})$ be closed range, $\left\{\Upsilon_\xi\right\}_{\xi \in \Omega}$ and $\left\{\Phi_\xi\right\}_{\xi \in \Omega}$ be c-K $K_1-g$-frames for $\mathcal{H}$ with respect to $\left\{\mathcal{\mathcal{H_\xi}}\right\}_{\xi \in \Omega}$. Assume that $K_2 \in \operatorname{End}_\mathcal{A}^*(\mathcal{K}), \Theta_1, \Theta_2 \in \operatorname{End}_\mathcal{A}^*(\mathcal{H},\mathcal{K})$ and $\Theta_1 \mathcal{T}_{\Upsilon} \mathcal{T}_{\Phi}^* \Theta_2^*+\Theta_2 \mathcal{T}_{\Phi} \mathcal{T}_{\Upsilon}^* \Theta_1^* \geq 0$.
\begin{enumerate}
\item[(i)]$P=\alpha_1 \Theta_1+\alpha_2 \Theta_2, \quad \quad \mathcal{R}\left(K_2\right) \subseteq \mathcal{R}(P),\mathcal{R}\left(P^*\right) \subseteq \mathcal{R}\left(K_1\right)$.
\item[(ii)]$\mathcal{Q}=\alpha_1 \Theta_1-\alpha_2 \Theta_2,\quad \quad \mathcal{R}\left(\mathcal{Q}^*\right) \subseteq \mathcal{R}\left(K_1\right), \mathcal{R}\left(K_2\right) \subseteq \mathcal{R}(\mathcal{Q})$ with $\overline{\mathcal{R}(\mathcal{Q}^{*})}$ is orthogonally complemented.
\end{enumerate}
Let $\alpha_1, \alpha_2>0,$ if one of (i), (ii) are holds then, $\left\{\alpha_1 \Upsilon_\xi \Theta_1^*+\right.$ $\left.\alpha_2 \Phi_\xi \Theta_2^*\right\}_{\xi \in \Omega}$ is a c- $K_2-g$-frame for $\mathcal{K}$ with respect to $\left\{\mathcal{\mathcal{H_\xi}}\right\}_{\xi \in \Omega}$.
\end{theorem}
\begin{proof}
Let $A_1, B_1$ and $A_2, B_2$ be frame bounds of $\left\{\Upsilon_\xi\right\}_{\xi \in \Omega}$ and $\left\{\Phi_\xi\right\}_{\xi \in \Omega}$, respectively. It is easy to show that, for every $\alpha_1, \alpha_2>0$ and $g \in \mathcal{H}$, 
$$\int_{\Omega}\left\langle\left(\alpha_1 \Upsilon_\xi \Theta_1^*+\alpha_2 \Phi_\xi \Theta_2^*\right) g,\left(\alpha_1 \Upsilon_\xi \Theta_1^*+\alpha_2 \Phi_\xi \Theta_2^*\right) g\right\rangle \mathrm{d} \nu(\xi)  \leq\left( \alpha_1^2 B_1\left\|\Theta_1^{*}\right\|^2+ \alpha_2^2 B_2\left\|\Theta_2^{*}\right\|^2\right)\left\langle  g, g\right\rangle .$$
On the other hand,
$$
\begin{aligned}
\int_{\Omega}\left\langle\left(\alpha_1 \Upsilon_\xi \Theta_1^*+\alpha_2 \Phi_\xi \Theta_2^*\right) g,\left(\alpha_1 \Upsilon_\xi \Theta_1^*+\alpha_2 \Phi_\xi \Theta_2^*\right) g\right\rangle \mathrm{d} \nu(\xi) & =\alpha_1^2 \int_{\Omega}\left\langle \Upsilon_\xi \Theta_1^* g,\Upsilon_\xi \Theta_1^* g \right\rangle \mathrm{d} \nu(\xi) \\
& +2 \alpha_1 \alpha_2\left\langle\left(\Theta_2 \mathcal{T}_{\Phi} \mathcal{T}_{\Upsilon}^* \Theta_1^*+\Theta_1 \mathcal{T}_{\Upsilon} \mathcal{T}_{\Phi}^* \Theta_2^*\right) g\right., g\rangle \\
  &+\alpha_2^2 \int_{\Omega}\left\langle \Phi_\xi \Theta_2^* g,\Phi_\xi \Theta_2^* g\right\rangle \mathrm{d} \nu(\xi) \\
& \geq \alpha_1^2 A_1\left\langle K_1^* \Theta_1^* g,\right\rangle+\alpha_2^2 A_2\left\langle K_1^* \Theta_2^* g,K_1^* \Theta_2^* g\right\rangle 
\end{aligned}
$$
Assume condition (ii) is true. we pose
$$
\lambda=\min \left\{A_1, A_2\right\},
$$
According to the parallelogram law, for every  $g \in \mathcal{H}_2$,
$$
\begin{aligned}
\alpha_1^2 A_1\left\langle K_1^* \Theta_1^* g,K_1^* \Theta_1^* g\right\rangle+\alpha_2^2 A_2\left\langle K_1^* \Theta_2^* g,K_1^* \Theta_2^* g\right\rangle & \geq \lambda\left(\left\langle \alpha_1 K_1^* \Theta_1^* g,\alpha_1 K_1^* \Theta_1^* g\right\rangle+\left\langle \alpha_2 K_1^* \Theta_2^* g,\alpha_2 K_1^* \Theta_2^* g\right\rangle\right) \\
& =\frac{\lambda}{2}\left(\left\langle K_1^*\left(\alpha_1 \Theta_1+\alpha_2 \Theta_2\right)^* g,K_1^*\left(\alpha_1 \Theta_1+\alpha_2 \Theta_2\right)^* g\right\rangle\right. \\
& \left.+\left\langle K_1^*\left(\alpha_1 \Theta_1-\alpha_2 \Theta_2\right)^* g,K_1^*\left(\alpha_1 \Theta_1-\alpha_2 \Theta_2\right)^* g\right\rangle\right) \\
& \geq \frac{\lambda}{2}\left\langle K_1^* \mathcal{Q}^* g,K_1^* \mathcal{Q}^* g\right\rangle \\& \geq \frac{\lambda}{2}\left\|K_1^{\dagger}\right\|^{-2}\left\langle \mathcal{Q}^* g,\mathcal{Q}^* g\right\rangle .
\end{aligned}
$$
Since $ \mathcal{R}(\mathcal{Q})\supseteq \mathcal{R}\left(K_2\right)$, consequently applying Lemma \textbf{1.8}, there exists $\alpha>0$ such that $$K_2 K_2^* \leq \alpha \mathcal{Q} \mathcal{Q}^*.$$ Hence for $g \in \mathcal{K}$, $$\left\langle \mathcal{Q}^* g,\mathcal{Q}^* g\right\rangle \geq \alpha^{-1}\left\langle K_2^* g,K_2^* g\right\rangle .$$ consequently, for every $g \in \mathcal{H}_2$, 
$$
\begin{aligned}
\frac{\lambda}{2} \alpha^{-1}\left\|K_1^{\dagger}\right\|^{-2}\left\langle K_2^* g,K_2^* g\right\rangle & \leq \int_{\Omega}\left\langle\left(\alpha_1 \Upsilon_\xi \Theta_1^*+\alpha_2 \Phi_\xi \Theta_2^*\right) g,\left(\alpha_1 \Upsilon_\xi \Theta_1^*+\alpha_2 \Phi_\xi \Theta_2^*\right) g\right\rangle \mathrm{d} \nu(\xi) \\
& \leq\left( \alpha_1^2 B_1\left\|\Theta_1^{*}\right\|^2+ \alpha_2^2 B_2\left\|\Theta_2^{*}\right\|^2\right)\left\langle  g, g\right\rangle .
\end{aligned}
$$
\end{proof}

\bibliographystyle{amsplain}

\end{document}